\documentclass{amsart}
\usepackage{amsmath,amssymb}
\title{Locally finite groups of finite centralizer dimension}
\date{21 April 2018; reworked 07 September 2018}
\author{Alexandre Borovik}\thanks{This is the Author Accepted Manuscript of the paper: A. Borovik and U. Karhum\"{a}ki, Locally finite groups of finite centralizer dimension. J. Group Theory (2019). Accepted for publication 22 January 2019. The manuscript will undergo copyediting, typesetting, and review of
the resulting proof before it is published in its final form.}
\address{School of Mathematics, University of Manchester, UK; alexandre $\gg {\rm at} \ll$ borovik.net}
\author{Ulla Karhum\"{a}ki}
\thanks{The second author is funded by the Vilho, Yrj\"{o} and Kalle V\"{a}is\"{a}l\"{a} Foundation of the Finnish Academy of Science and Letters.}

\address{School of Mathematics, University of Manchester, UK;
ukarhumaki $\gg {\rm at} \ll$ gmail.com}

\newtheorem{lemma}{Lemma}
\newtheorem{theorem}{Theorem}

\newtheorem{proposition}{Proposition}

\newtheorem{fact}{Fact}

\usepackage{color}

\begin{document}

\maketitle
\begin{abstract}
We describe structure of locally finite groups of finite centralizer dimension.
\end{abstract}

\section{Introduction}

Recall that the centralizer dimension of a group is the maximum length of a chain of nested
centralizers. This paper describes structure of locally finite groups of finite centralizer dimension; the result will be used by the second author in \cite{Karhumaki2017B}.

\begin{theorem}\label{th:structure-lfg} Let $G$ be a locally finite group of finite centralizer dimension $c$. Then

 $G$ has a  normal series
\[
1 \unlhd S \unlhd L  \unlhd  G,
\]
where

\begin{itemize}

\item[{\rm (a)}] $S$ is solvable of derived length bounded by a function of $c$.

\item[{\rm (b)}] $\overline{L} =L/S$ is a direct product $\overline{L} = \overline{L}_1 \times \cdots \times \overline{L}_m$ of finitely many non-abelian simple groups.

\item[{\rm (c)}] Each $\overline{L}_i$ is either finite, or a Chevalley group, or a twisted analogue of a Chevalley group,  or one of the non-algebraic twisted groups of Lie type $^2B_2$, $^2F_4$, or $^2G_2$, over a locally finite field.
\item[{\rm (d)}]  The factor group $G/L$ is finite.

\end{itemize}
\end{theorem}

Our proof uses the Classification of Finite Simple Groups.

\section{Proof of Theorem \ref{th:structure-lfg}}

Our theorem absorbs a number of known results, and we shall handle the proof issue by issue.

We work with the group $G$ which satisfies the assumptions of Theorem \ref{th:structure-lfg}.

\subsection{Control of sections}

If $G$ is a group and $H \leq G$ and $K\lhd H$, then the factor group $H/K$ is called a \emph{section} of $G$.

One immediate difficulty encountered in any study of groups of finite centralizer dimension is that the descending chain condition for centralizers is inherited by subgroups of $G$, but not by factor groups or sections. This happens even in the class of periodic nilpotent groups of finite centralizer dimension (and they are of course locally finite), an example can be found in \cite[Section 4]{BRYANT1979371}.

So we need a sufficiently strong property which holds in every locally finite group of finite centralizer dimension and is inherited by its sections. Luckily, this property is provided by the following result.

\begin{fact}[Khukhro \cite{khukhro_2009}]\label{khukhro} Periodic locally solvable groups of finite centralizer dimension are solvable and have derived length bounded by function of centralizer dimension.
\end{fact}

Let us call a group $G$ \emph{constrained} if derived lengths of its solvable subgroups are bounded.  In view  of Fact \ref{khukhro},  periodic groups of finite centralizer dimension are constrained.

\begin{lemma}\label{lm:constrained} Let $G$ be a locally finite group of finite centralizer dimension and $\overline{H}=H/K$ its section. Then $\overline{H}$ is constrained.
\end{lemma}

\begin{proof} By Fact \ref{khukhro} we may assume that every solvable subgroup in $G$ has derived length at most $d$. Since $\overline{H}$ is locally finite, it suffices to prove that an arbitrary finite solvable subgroup $\overline{S}$ in  $\overline{H}$ has derived length at most $d$. Pick representatives $s_1,\dots s_n$ of cosets of $\overline{S}$ of $K$ and generate by them a subgroup $R$; it is a finite subgroup, and by the well-known Frattini Argument for finite groups, $R$ contains a subgroup $P$ such that $P(R\cap K) = R$ and $P\cap K$ is nilpotent. Thus, $P$ is a solvable subgroup of $G$, hence has derived length at most $d$; but $PK/K = RK/K = \overline{S}$, hence the derived length of $\overline{S}$ is also at most $d$.

\end{proof}

\subsection{Simple sections in locally finite groups of finite centralizer dimension}

The following result by Brian Hartley (based on the Classification of Finite Simple Groups) allows us to identify constrained simple locally finite groups.

\begin{fact}[Hartley \cite{Hartley1995}]\label{hartley} Let $L$ be an infinite simple locally finite group. If  some finite group is not involved in $L$, then $L$ is a Chevalley group, or a twisted version of a Chevalley group, or one of the non-algebraic twisted groups of Lie type $^2B_2$, $^2F_4$, or $^2G_2$ over a locally finite field.
\end{fact}

We shall call groups in conclusion of Fact \ref{hartley} \emph{simple groups of Lie type}.

As an immediate corollary, we have the following theorem.

\begin{theorem} \label{th:simple-sections}
If  $L$ is an infinite simple section of a locally finite group of finite centralizer dimension, then $L$ is of Lie type.
\end{theorem}

It is a partial generalisation of the result by Simon Thomas (1983)  (which is
also based on the Classification of Finite Simple Groups).

\begin{fact} {\rm (Thomas \cite{thomas21983})}
An infinite simple locally finite group which satisfies the minimal condition on centralizers is of Lie type over a locally finite field.
\end{fact}

\subsection{Quasisimple simple locally finite groups of Lie type}

Recall that a group $H$ is called \emph{quasisimple} if $H = [H,H]$ and $H/Z(H)$ is a non-abelian simple group.

\begin{proposition} \label{prop:multiplier}
If $H$ is a quasisimple locally finite group and $H/Z(H)$ is a simple group of Lie type than $|Z(H)|$ is finite and bounded by a constant depending only on $H/Z(H)$.
\end{proposition}

\begin{proof}
The group of Lie type $\overline{H} =H/Z(H)$ is defined over a locally finite field, say $F$, and taking groups of points  in $\overline{H}$ of finite subfields of $F$, we can construct a sequence of subgroups
\[
1 < H_1 < H_2 < \dots
\]
such that their images $\overline{H}_i$ in $\overline{H}$ are simple groups of the same Lie type as  $\overline{H}$, $H_i= [H_i,H_i]$ for all $i =1,2,\dots$, and  $\overline{H} = \bigcup_{i=1}^\infty \overline{H}_i$. Let $Z=Z(H)$. The subgroups $Z(H_i) = H_i \cap Z$ are factor groups of the Schur multipliers of groups of Lie type $\overline{H}_i$ and have orders of bounded size (see \cite[\S 6.1]{GLS}). Therefore the sequence of groups
\[
1 \leq Z(H_1) \leq Z(H_2) \leq \dots
\]
stabilizes at some finite group $Z_*$ of bounded order. If $Z$ is infinite, there is an element $z \in Z \smallsetminus Z_*$ which is written as a product of some commutators from $H$: $z =[h_1,h_2]\cdots [[h_{2k-1},h_{2k}]$ with elements $h_1, h_2, \dots, h_{2k}$ contained in one of the  subgroups $H_m$; but then $z \in [H_m,H_m] = H_m$ and therefore belongs to $Z(H_m) \leq Z_*$ -- a contradiction.
\end{proof}

\subsection{Composition factors and composition series}

Now we have to address another difficulty: the concept of a composition factor is somewhat vague in the case of locally finite groups because the classical Jordan-H\"{o}lder Theorem for composition series of finite groups is no longer true; a spectacular example is given by Brian Hartley \cite[pp.~2--3]{Hartley1995}. Moreover, there are countably infinite locally finite simple groups for which the following holds: $G$ possess a  series
\[
\ldots G_{-2} \triangleleft G_{-1} \triangleleft G_{0}
\triangleleft G_{1} \triangleleft G_{2} \ldots
\]
of proper subgroups of $G$ where the factors are finite and $G =\bigcup_{i=1}^\infty G_i$ (\cite[Theorem 1.27]{Hartley1995} and \cite{Meierfrankenfeld1995}).

For that reason we formulate the theorem by Buturlakin and Vasil'ev  \cite{Buturlakin2013} in a weaker form more suitable for our use.

\begin{fact} \label{fact:Buturlakin} {\rm(Buturlakin and Vasil'ev  \cite{Buturlakin2013})}
Let $G$ be a locally finite group of centralizer dimension $c$ and
\[
1=G_0 < G_1 < G_2 < \dots < G_l = G
\]
a finite subnormal series in $G$. Then the number of distinct non-solvable factors $G_i/G_{i-1}$,  $i \geqslant 1$, is at most $5c$.
\end{fact}

\subsection{Proof of Theorem  \ref{th:structure-lfg}: solvable radical and the layer}

We start building the normal series
\[
1 \unlhd S \unlhd L  \unlhd  G
\]
of Theorem  \ref{th:structure-lfg}.

For a constrained group $H$, we denote by $R(H)$ the maximal locally solvable normal subgroup of $H$; note that $R(H)$ is solvable and $H/R(H)$ has no non-trivial solvable normal subgroups. We also denote by $Q(H)$ the minimal normal subgroup of $H$ with solvable factor $H/Q(H)$; since $H$ is constrained, $Q(H)$ exists and coincides with the last term of the derived series of $H$. Also, $Q(H)$ has no non-trivial solvable factor groups. We shall call $H$ \emph{truncated} if $R(H) = 1$ and $Q(H)=H$. Observe also that $R(H)$ and $Q(H)$ are characteristic subgroups of $H$.

We start with the normal series
\[
1 = G_0 \lhd G_1 =G,
\]
and refine and re-build it and get subsequent series
 \[
1 = G_0 \lhd G_1 < \dots < G_l =G,
\]
appropriately changing numeration at every step, in accordance with the following rules:

\begin{itemize}
\item For every  factor $G_i/G_{i-1}$ that is not truncated, we insert subgroups
\[
G_{i-1} \unlhd G_j \unlhd G_k \unlhd G_i,
\]
where $G_j$ in the full preimage of $R(G_i/G_{i-1})$ and $G_k$ is the full preimage of $Q(G_i/G_{i-1})$.
\item If $G_i/G_{i-1}$ and $G_j/G_{j-1}$, $i < j-1$, are two non-trivial truncated factors and there are  no truncated factors  $G_k/G_{k-1}$ for $i < k<j$, then $G_{j-1}/G_i$ is solvable and we can remove from the series its members $G_{i+1},\dots, G_{j-2}$.
\item If $G_i/G_{i-1}$ is a truncated factor and there is a normal subgroup $G_j \lhd G$ fitting into  $G_{i-1} < G_j < G_i$, we insert it in the series -- and repeat the process from the beginning.
\end{itemize}

In view of Fact~\ref{fact:Buturlakin}, the process terminates after finitely many steps, producing a finite series  \(
1 = G_0 \lhd G_1 < \dots < G_l =G
\) of normal subgroups where every factor is solvable or truncated without non-trivial proper characteristic subgroups.

 We again apply Fact~\ref{fact:Buturlakin}:

\begin{lemma} In the series above, truncated factors $G_i/G_{i-1}$  are finite direct products of isomorphic non-abelian simple groups which are either finite or of Lie type.
\end{lemma}

\begin{proof} Consider normal series in $H =  G_i/G_{i-1}$;  the number of non-solvable factors is each series is bounded by  Fact~\ref{fact:Buturlakin}. Consider a normal series
 \[
 H = H_1 \rhd H_2 \rhd H_3 \rhd \cdots \rhd H_j \rhd \cdots
 \]
with maximal possible number of nonsolvable factors.   If $H_k/H_{k+1}$ is the last non-solvable factor then $H_{k+1}$ is solvable, hence $H_{k+1} =1$ since $H$ is truncated. Moreover, if  $H_k \rhd K \ne 1$ and $H \rhd K$, then $H_k/K$ is solvable -- for otherwise
the refinement
 \[
 H = H_1 \rhd H_2 \rhd H_3 \rhd \cdots \rhd H_j \rhd K \rhd 1
 \]
 would contain more non-solvable factors than the originally chosen series. Without loss of generality we can replace $K$ by the last term of the derived series of $H_k$. Now $K$ is truncated and the same argument as the one applied to $H_k$ shows that $K$ contains no nontrivial proper subgroups normal in $H$, that is, $K$ is a minimal normal subgroup in $H$.

Let now $K_1, K_2, \dots$ be distinct minimal normal subgroups in $H$ and $K = \prod_j K_j$ their product. Obviously, $K$ is a characteristic subgroup of $H$ and therefore $K=H$. Since all $K_j$ are minimal normal and non-abelian, any product $K_1K_2\cdots K_l$ of finitely many of them is a direct product, $K_1K_2\cdots K_l = K_1 \times K_2 \times \cdots \times K_l$. Obviously, $K_1 \lhd K_1K_2 \lhd K_1K_2K_3 \lhd \dots \lhd H$; this normal series in $H$ extends to a subnormal series in $G$ with non-solvable factors, and, in view of Fact~\ref{fact:Buturlakin}, is finite. Therefore $K$ is a direct product $H = \bigoplus_j K_j$ of finitely many subgroups $K_j$;   since $H$ is characteristically simple, all $K_j$s are isomorphic. If $K_j$ is infinite, than it is of Lie type by Theorem \ref{th:simple-sections}.
\end{proof}

Returning to the proof of Theorem \ref{th:structure-lfg}, we can now  define $S = R(G)$ and coherently define $L$ as the full pre-image in $G$ of the \emph{layer} $L(G/S)$, that is, the product of all simple subnormal subgroups (\emph{components}) of $G/S$. Therefore $\overline{L} =L/S$ is a direct product $\overline{L} = \overline{L}_1 \times \cdots \times \overline{L}_m$ of finitely many non-abelian simple groups. This proves Theorem \ref{th:structure-lfg} as soon as we do the point (d) of the Theorem~\ref{th:structure-lfg}, finiteness of $G/L$ -- see the following Sections.

\subsection{Action of $G$ on $G/S$}

We retain notation from the previous Section. Set $\overline{G} = G/S$.

Our first observation is that $C_{\overline{G}}(\overline{L}) \cap \overline{L} = 1$ and  $C_{\overline{G}}(\overline{L}) \unlhd \overline{G}$. If $C_{\overline{G}}(\overline{L}) \ne 1$, then, applying analysis of the previous Section to the full preimage of $C_{\overline{G}}(\overline{L})$ in $G$,  we see that  the group $C_{\overline{G}}(\overline{L})$ has subnormal non-abelian simple subgroups which are subnormal in $\overline{G}$ but do not belong to  $\overline{L}$, which contradicts to the way $\overline{L}$ was constructed. Hence $C_{\overline{G}}(\overline{L}) = 1$.

Now the group $G$, in its action on $\overline{L}$ by conjugation, permutes simple subgroups $\overline{L}_i$; the kernel of this permutation action, say $G^\circ$, is a normal subgroup of finite index in $G$. Without loss of generality, we can assume that $G^\circ = G$ and each $\overline{L}_i \unlhd \overline{G}$.

Let now $\overline{M}=\overline{L}_1 \times \cdots\times \overline{L}_k$ be the product of all \emph{infinite} components of $\overline{G}$; if $\overline{M}=1$, then $\overline{L}$ and hence $\overline{G}$ are finite, thus the point (d) holds and Theorem~\ref{th:structure-lfg} is proven.

So we can assume that $\overline{M} \ne 1$. Denote by $\overline{N} = \overline{L}_{k+1} \times \cdots\times \overline{L}_m$ the product of all finite components of $\overline{G}$. If  $\overline{N} \ne 1$, then $C_{\overline{G}}(\overline{N})$ is the kernel of the action of $\overline{G}$  on $\overline{N}$ by conjugation and therefore has finite index in $\overline{G}$. Again, we can assume without loss of generality that $C_{\overline{G}}(\overline{N})= \overline{G}$ and  $\overline{N} = 1$, that is, all components of $\overline{G}$ are infinite simple groups of Lie type over (infinite) locally finite fields.

\subsection{The factor group $G/L$ is abelian-by-finite} \label{sec:ab-by-fi}
We turn our attention to the action of $\overline{G}$ on $\overline{L}$.

It is well-known that every automorphism of a group of Lie type over a locally finite field $F$, say $X=X(F)$, is a product of inner, diagonal, graph, and field automorphisms. If $\rm{Out}\, X = {\rm Aut}\, X/ \rm{Inn}\, X$ is the group  of outer automorphisms of $X$, then images in $\rm{Out}\, X$ of diagonal and graph automorphisms of $X$ generate a finite subgroup. Also, the image $\Gamma$ in $\rm{Out}\, X$ of the group of field automorphisms is naturally isomorphic to the group ${\rm Aut}\, F$. It is well known that  ${\rm Aut}\, F$ is a factor group of $\widehat{\mathbb{Z}}$,  the profinite completion of the additive group of integers (the latter is the Galois group of the algebraic closure of a finite prime field).

We have a natural embedding
\[
\overline{G}/\overline{L} \hookrightarrow \prod_{i=1}^m {\rm Out}\, \overline{L}_i.
\]
We see now that if $\overline{G}/\overline{L}$ is infinite, it contains an abelian subgroup $\Delta$ of finite index which either centralizes, or acts by field automorphisms on components $\overline{L}_i$ -- in the sense that elements from the preimage $\overline{D}$ of $\Delta$ induce on $\overline{L}_i$ automorphisms that are products of inner and field automorphisms.

Let $D$ be the full preimage of $\overline{D}$  in $G$, then $D$ has a finite index in $G$. So for the rest of the proof we can assume, without loss of generality, that $G/L$ is abelian and outer automorphism induced from $\overline{D}$ on Lie type subgroups $\overline{L}_i$ are field automorphisms.

\subsection{Frattini Argument}

At this point it becomes essential to find a more efficient way around the fact that the descending chain condition for centralizers in general is not preserved under taking factor groups.

The general situation is the following: we have a group $G$ of finite centralizer dimension and a subgroup $K \lhd G$; we wish to derive some information about $\widehat{G} = G/K$ without being able to prove directly that $\widehat{G}$ has finite centralizer dimension.
The idea is to calculate instead in an appropriate
partial complement $M \leq G$ to $K$, that is, a subgroup such that $G = MK$ and $M$ is sufficiently small for easier deduction of the desired facts about $G/K \simeq M/(M\cap K)$.

The classical way to construct partial complements is the Frattini Argument. For a set of prime numbers $\pi$, a periodic group $H$ is called a $\pi$-group (correspondingly, $\pi'$-group), if prime factors of orders of elements from $H$ belong (correspondingly, do not belong) to $\pi$. A Sylow $\pi$-subgroup of a group $H$ is a maximal $\pi$-subgroup of $H$.

\begin{lemma} \label{Lemma:FrattinI-Argument} {\rm (Frattini Argument)}
Let $H$ be a locally finite group and $K \lhd H$ a normal subgroup. Assume that the Sylow $\pi$-subgroups in $K$ are conjugate in $K$ and $P$ is one of them. Then $H = K N_H(P)$.
\end{lemma}

\begin{proof} Proof is exactly the same as the well-known proof in the finite case.
\end{proof}

The value of the Frattini Argument is obvious because of following important fact.

\begin{fact} {\rm (Bryant and Hartley \cite[Theorem 1.6]{Bryant-Hartley-1979})}
In a periodic solvable group $H$ with descending chain condition for centralizers and for all set of primes $\pi$, Sylow $\pi$-subgroups are conjugate.
\end{fact}

We shall start now using the Frattini argument as a tool for carving out from $G$ subgroups where we have better control of centralizers.

\subsection{Toward finiteness of $G/L$} \label{sec:towards-finiteness}
Let us now look at the group $\Delta < \overline{G}/\overline{L}$ constructed in Section \ref{sec:ab-by-fi} and its full preimage $\overline{D}$ in $\overline{G}$.

To prove that $\overline{G}/\overline{L}$ is finite it will suffice to prove that $\Delta$ is finite. So let us assume that $\Delta$ is infinite and work towards a contradiction. Among components $\overline{L}_i \lhd \overline{L}$, we pick one, say $\overline{K}$, such that $\overline{D}$ induces in its action on $\overline{K}$ an infinite group of outer automorphisms. The group  $\overline{K}$ is of Lie type over a locally finite field $F$; let $p$ be the of characteristic of $F$. Consider the natural homomorphism
\[
\rho: \overline{D} \longrightarrow \overline{K} \rtimes \Gamma, \quad \mbox{ where } \quad \Gamma = {\rm Aut}\, F,
\]
and take ${\rm E} = {\rm Im}\,\rho \cap \Gamma$.

 The group $\Gamma$  is the continuous image of $\widehat{\mathbb{Z}}$,
and ${\rm E}$, as a locally finite subgroup of $\Gamma$, is locally cyclic and  is a direct sum of finite cyclic
groups of pairwise coprime prime power orders; let $\epsilon_1 ,\epsilon_2,\dots  $, be generators of the cyclic direct
summands of ${\rm E}$, then we have an infinite decreasing sequence of fields
\[
C_F(\langle\epsilon_1\rangle) > C_F(\langle\epsilon_1,\epsilon_2\rangle) > C_F(\langle\epsilon_1,\epsilon_2,\epsilon_3\rangle) > \dots,
\]
and correspondingly an infinite descending chain of centralizers
\[
C_{\overline{K}}(\langle\epsilon_1\rangle) > C_{\overline{K}}(\langle\epsilon_1,\epsilon_2\rangle) > C_{\overline{K}}(\langle\epsilon_1,\epsilon_2,\epsilon_3\rangle) > \dots.
\]
This would produce a contradiction if we knew that $\overline{G}$ had a descending chain condition for centralizes; but we don't, so we have to make some further surgery on the group $G$, and, in particular, cut ${\rm E}$ to a manageable size.

Pick in ${\rm E}$ elements $\alpha_1, \dots, \alpha_n$, of pairwise different prime orders $p_1,\dots, p_n$, none of which is $p$, and none divides the order of the center of any quasisimple extension  of $\overline{K}$ (see Proposition~\ref{prop:multiplier}), making sure that $n$ is bigger than the centralizer dimension $c$ of $G$. Take their product $\alpha = \alpha_1\alpha_2\cdots\alpha_n$ and its preimage $\bar{a}$ in $\overline{G}$; let $a \in G$ be some coset representative of $\bar{a}$ and $A = \langle a \rangle$ is the cyclic group generated by $a$; replacing $a$ by another coset representative, we can ensure that prime divisors of $|A|$ are exactly $p_1,\dots, p_n$.

Let now $K$ be the full preimage of $\overline{K}$; we can replace, without loss of generality, $G$ by $KA$; the solvable radical $S$ could slightly grow up by absorbing part of $A$, but this does not affect our considerations; we still have the property that, for appropriate powers $a_1,\dots, a_n$ of $a$, we have a descending chain of centralizers
\[
C_{\overline{K}}(\langle a_1\rangle) > C_{\overline{K}}(\langle a_1,a_2\rangle) > \dots >  C_{\overline{K}}(\langle a_1,a_2,\dots, a_n\rangle).
\]

\subsection{Trimming the solvable radical}

Let $U$ be a Sylow $p$-subgroup of $S$, where $p$ is the characteristic of $F$, the underlying field of $\overline{K}$. Using Frattini Argument, we replace $G$ by $N_G(U)$ and assume, without loss of generality, that $U \lhd G$. Now we take a Sylow $p'$-subgroup $Q$ in $S$ and, applying the Frattini Argument again, replace $G$ by $N_G(Q)$, thus assuming, without loss of generality, that $Q\lhd G$. Now $S = U \times Q$.

Let now $P$ be a Sylow $p$-subgroup in $K$, then its image $\overline{P}$ in $\overline{K}$ is a maximal unipotent subgroup in the group $\overline{K}$ of Lie type over an infinite field of characteristic $p$; hence $\overline{P}$ contains an infinite elementary abelian subgroup $\overline{V}$ (the center of any root subgroup can be used for that purpose). We denote by $V$ a Sylow $p$-subgroup in the full preimage of $\overline{V}$ in $K$, then $V$ acts on $Q$ by conjugation and $U$ belongs to the kernel of this action. The group $VQ$ has finite centralizer dimension, and therefore has uniformly bounded lengths of  chains of centralizers of subsets from $\overline{V}$ in $Q$. At this point we can invoke the following  fact.

\begin{fact}  If an elementary abelian $p$-group $W$ of order~$p^n$ acts faithfully on a locally finite solvable $p'$-group $Q$, then there exists a series of subgroups $$W=W_0>W_1>W_2>\dots>W_n=1$$ such that $$C_Q(W_0)<C_Q(W_1)<\dots<C_Q(W_n).$$
\end{fact}

\begin{proof}
This fact is proven by Khukhro \cite[Lemma 3]{khukhro_2009} in the special case of $Q$ being finite and nilpotent, and, in particular, when $Q$ is abelian. Reduction of the more general case to this one is easy. First of all, $W$ acts on $Q$ faithfully, therefore for every element $w \in W \smallsetminus \{1\}$ we can pick an element $g_w \in Q$ such that $g^w_w \ne g^w$ and replace $Q$ by a $W$-invariant finite subgroup $\langle g_w^W : w \in W \smallsetminus \{1\}\rangle$ generated by the $W$-orbits $q_w^W$ of elements $g^w$. So we can assume without loss of generality that $Q$ is finite. Now we can apply a much more general  result:

\begin{fact}
 {\rm (Hartley-Turull, \cite[Theorem 3.31]{Isaacs2008})}. Let $W$ act via automorphisms on $Q$, where
$W$ and $Q$ are finite groups, and suppose that $(|W|, |Q|) = 1$. Assume also that at least one of $W$ or $Q$ is solvable. Then $W$ acts via automorphisms on some
abelian group $R$ in such a way that every subgroup $U \subseteq W$ has equal
numbers of fixed points on $Q$ and on $R$.
\end{fact}
Now an application of Khukhro's lemma \cite[Lemma 3]{khukhro_2009} to  the action of $W$ on $R$ completes the proof.
\end{proof}

We can now return to proof of Theorem  \ref{th:structure-lfg}. We now see that $Y=C_V(Q)$ has finite index in $V$ and therefore the image $\overline{Y}$ of $Y$ in $\overline{K}$ is infinite. Let $H=C_K(Q)$; obviously, $H \unlhd K$ and contains  $\overline{Y}$;  but $\overline{K}$ is simple, hence $\overline{K}=\overline{H}$  and contains  $\overline{Y}$. So, without loss of generality, we can replace $K$ by $H=C_K(Q)$, and then replace $G$ by  $KA$.

Now $Q \leq Z(K)$. At the last step of our trimming procedure, we replace $K$ by the intersection of its derived series, making sure that $K= [K,K]$, and then again replace $G$ by $KA$.

Consider $\widehat{K} = K/U$ and let $\widehat{Q}$ be the image of $Q$ in $\widehat{K}$, then  $\widehat{Q} \leq  Z(\widehat{K})$. The centers of locally finite quasisimple groups of Lie type are finite by Proposition \ref{prop:multiplier}, hence $\widehat{Q}$ and $Q$ are finite.
In the resulting normal series
\[
1 \unlhd U \unlhd S \lhd K \unlhd G
\]
the factor $S/U$ is a finite abelian group in view of Proposition \ref{prop:multiplier}. In particular, only finitely many distinct prime numbers divide orders of elements in $S$.

\subsection{End of proof}
Now we can return to the chain of centralizers
\[
C_{\overline{K}}(\langle a_1\rangle) > C_{\overline{K}}(\langle a_1,a_2\rangle) > \dots > C_{\overline{K}}(\langle a_1,a_2,\dots, a_n\rangle)
\]
constructed in Section \ref{sec:towards-finiteness}.

Recall that $A$ was constructed in a way that $a_i \in A$ have orders coprime to orders of elements in $U$ and in $S/U$. These restrictions have been forced on $A$ with the aim of lifting centralizers of subgroups $B \leq A$ in $\overline{K}$ to centralizers in $K$, that is, proving that $C_{\overline{K}}(B)$ is the image of $C_K(B)$ in $\overline{K}$. For that, we need a simple tool from finite group theory.

\begin{fact} \label{fact-lifting-centralisers}
Let $B$ be a finite cyclic $\pi$-group of automorphisms of a finite group $H$ and  and let $R$  an $B$-invariant normal $\pi'$-subgroup of $H$. Then $C_{H/R}(B)$ is the image of $C_H(B)$ in $H/R$.
\end{fact}

\begin{proof} This is a special case of \cite[Corollary 3.28]{Isaacs2008} \end{proof}

Now we can expand, in a routine way, Fact \ref{fact-lifting-centralisers} to locally finite groups.

\begin{lemma}
Let $B$ be a finite cyclic $\pi$-group of automorphisms of a locally finite group $K$ and  let $S$  an $B$-invariant normal solvable $\pi'$-subgroup of $H$. Assume, in addition, that orders of elements from $S$ are divisible only by finitely many different prime numbers.  Then $C_{K/S}(B)$ is the image of $C_K(B)$ in $K/P$.
\end{lemma}

Now, after lifting centralizers in the chain
\[
C_{\overline{K}}(\langle a_1\rangle) > C_{\overline{K}}(\langle a_1,a_2\rangle) > \dots > C_{\overline{K}}(\langle a_1,a_2,\dots, a_n\rangle),
\]
from $\overline{K}$ to $K$, we have the chain of centralizers
\[
C_{{K}}(\langle a_1\rangle) > C_{{K}}(\langle a_1,a_2\rangle) > \dots > C_{{K}}(\langle a_1,a_2,\dots, a_n\rangle),
\]
of the length exceeding the centralizer dimension of $G$.

This contradiction completes the proof of Theorem~\ref{th:structure-lfg}.\hfill $\square$ $\square$

\section*{Acknowledgements}

We thank the anonymous referee for useful comments and suggestions.

\end{document}